\documentclass[11pt]{amsart}
\usepackage{amsfonts,amsmath,amssymb,amsthm,cmtiup} 
\usepackage{mathrsfs}
\usepackage[noadjust]{cite}

\newtheorem{theorem}{Theorem}
\newtheorem{corollary}{Corollary} 
\newtheorem{proposition}{Proposition} 
\newtheorem{fact}{Fact}

\theoremstyle{definition} 
\newtheorem{remark}{Remark}

\begin{document}

\title[Topological Groups with Strong Disconnectedness Properties]{Topological Groups\\ with Strong 
Disconnectedness Properties} 

\author{Ol'ga Sipacheva}

\begin{abstract}
Topological groups whose underlying spaces are basically disconnected, $F$-, or $F'$-spaces but not $P$-spaces 
are considered. It is proved, in particular, that the existence of a Lindel\"of basically 
disconnected topological group which is not a $P$-space is equivalent to the existence of 
a Boolean basically disconnected Lindel\"of group of countable pseudocharacter, that free and free Abelian 
topological groups of zero-dimensional non-$P$-spaces are never $F'$-spaces, and that the existence of a free 
Boolean $F'$-group which is not a $P$-space is equivalent to that of selective ultrafilters on $\omega$. 
\end{abstract}

\keywords{Free topological group, free Abelian topological group, free Boolean topological 
group, basically disconnected group, $F$-group, $F'$-group, $P$-space, selective ultrafilter}

\subjclass[2020]{54H11, 54G05, 03E35}

\address{Department of General Topology and Geometry, Faculty of Mechanics and  Mathematics, 
M.~V.~Lomonosov Moscow State University, Leninskie Gory 1, Moscow, 199991 Russia}

\email{o-sipa@yandex.ru, osipa@gmail.com}

\maketitle

There is a whole hierarchy of classical strong disconnectedness properties: a space $X$ is \emph{maximal} if it 
has no isolated points and any two disjoint subsets of $X$ have disjoint closures; $X$ is 
\emph{extremally disconnected} if any two disjoint open subsets of $X$ have disjoint closures (or, equivalently, 
the closure of any open set in $X$ is open); and $X$ is \emph{basically disconnected} if the closure of any 
cozero set in $X$ is open. This list is naturally continued with $F-$ and $F$-spaces: $X$ 
an \emph{$F$-space} if any two disjoint cozero sets are completely (=functionally) separated in $X$; and, 
finally, $X$ is an $F'$-space if any two disjoint cozero sets in $X$ have disjoint closures. Clearly, each of 
these properties (except maximality) is a relaxation of the preceding one. Abusing terminology, we will refer to 
all spaces listed above as ``strongly disconnected,'' although $F$- and $F'$-spaces may be connected.

It is well known that all strong disconnectednesses badly affect homogeneity properties (for example, a 
homogeneous space strongly disconnected in any of the above senses cannot contain an infinite compact subspace; 
see, e.g., \cite{R20}). Thus, it is natural to ask whether any of them can coexist with the property of being a 
topological group, which can be regarded as ultimate homogeneity. A ZFC-consistent answer was given by Malykhin, 
who proved the existence of many nondiscrete maximal topological groups under the assumption $\mathfrak 
p=\mathfrak c$ \cite{Malykhin1975}. Thus, the problem is: Does there exist in ZFC a nondiscrete strongly 
disconnected (in one of the above senses) topological group? or, more generally, under what assumptions does 
there exist a nondiscrete strongly disconnected topological group?

For basically disconnected groups, this problem has been more or less solved. As 
mentioned, in \cite{Malykhin1975} Malykhin constructed a consistent example of a nondiscrete maximal group under 
the assumption $\mathfrak p=\mathfrak c$, and in \cite{Malykhin1979} he proved that any maximal group must 
contain an open countable maximal subgroup. On the other hand, in \cite{RS} Reznichenko and the author 
proved that the existence of a countable nondiscrete maximal (or even only extremally disconnected) 
topological group implies the existence of rapid ultrafilters, and in \cite{Protasov} 
Protasov proved that it implies the existence of $P$-point ultrafilters (see also 
\cite[Corollary~5.21]{Zelenyuk}). Thus, the nonexistence of maximal groups is consistent with ZFC. 

The existence in ZFC of extremally disconnected groups is Arkhangelskii's celebrated 1967 problem. 
It has been solved (in the negative) for countable groups \cite{RS}, but the uncountable case still remains 
open. 

The situation with basically disconnected groups and groups which are $F$- or $F'$-spaces is different. On the 
one hand, clearly, in the class of countable spaces, all strong disconnectednesses (except maximality) are 
equivalent, so that nondiscrete strongly disconnected countable groups cannot exist in ZFC. However, 
since all cozero sets in a $P$-space are obviously clopen, it follows that any topological $P$-group is basically 
disconnected (and hence an $F$- and an $F'$-space). Thus, the correct question is: Does there exist in ZFC a 
topological group whose underlying space is basically disconnected (an $F$-space, 
an $F'$-space) but not a $P$-space? Note that all maximal 
$P$-groups are discrete and the existence of a nondiscrete extremally disconnected $P$-group is 
equivalent to that of measurable cardinals (see~\cite{A1}).

Yet another distinguishing feature of extremally disconnected groups is that any such group must contain an open 
Boolean subgroup, i.e., a subgroup in which all elements are of order~2 \cite{Malykhin1975}. 
This reduces the existence problem for extremally disconnected 
groups to the case of Boolean groups. However, basically disconnected groups, even those not 
being $P$-spaces, do not have this property: for example, if $G$ is a nondiscrete countable extremally 
disconnected group (which consistently exists) and $H$ is an arbitrary nondiscrete $P$-group, then $G\times H$ is 
basically disconnected \cite{Hindman} but not necessarily contains an open Boolean group.

In this paper we show that, nevertheless, the existence problem for paracompact 
finite-dimensional $F$-groups of countable pseudocharacter does reduce to the case of Boolean groups. 
We also prove that (1)~a free (or free Abelian) topological group is basically disconnected if and only if it is 
a $P$-space; (2)~for any Tychonoff space $X$,  the following conditions are equivalent: \textup{(i)}~the 
free topological group of $X$ is an $F'$-space, \textup{(ii)}~the free Abelian topological group of $X$ is an 
$F'$-space, \textup{(iii)}~$X$ is a $P$-space; and (3)~the existence of a free Boolean topological $F'$-group  
which is not a $P$-space is equivalent to the existence of a selective ultrafilter on~$\omega$. 

\vspace{3pt}

Throughout the paper by a space we mean a Tychonoff (= completely regular Hausdorff) topological space, unless 
otherwise stated, and assume all topological groups under consideration to be Hausdorff. When considering a 
group, we denote its identity (or zero, if the group is Abelian) element by 1 (by 0). 

A  subset of a space is a \emph{$P$-set} if every $G_\delta$-set containing it is a
neighborhood of it. A space in which every singleton is a $P$-set (or, equivalently, all $G_\delta$-sets are 
open) is called a \emph{$P$-space}. By a \emph{$P$-group} (an \emph{$F$-group}, an \emph{$F'$-group}) we mean a 
topological group whose underlying space is a $P$-space (an $F$-space, an $F'$-space). 

Let $X$ and $Y$ be arbitrary (not necessarily completely regular Hausdorff) topological spaces. A continuous 
surjection $p\colon X\to Y$ is said to be \emph{$\mathbb R$-quotient} if the continuity of any 
function $\varphi\colon Y\to \mathbb R$ is equivalent to the continuity of the composition $\varphi\circ p$, or,  
in other words, the topology of $Y$ is the finest completely regular topology with respect to which $p$ is 
continuous. In this case, $Y$ is called an \emph{$\mathbb R$-quotient space} of $X$ with 
respect to $p$, and its topology is called the \emph{$\mathbb R$-quotient topology}. For any 
topological space $X$ and any surjection $p\colon X\to Y$ onto a set $Y$, there exists a 
unique $\mathbb R$-quotient topology on $Y$~\cite{R-quotient}. Clearly, if $X$ and $Y$ are Tychonoff spaces and 
$q\colon X\to Y$ is a quotient map, then the $\mathbb R$-quotient topology on $Y$ with respect to $q$ coincides 
with the quotient topology. 

We recall that a \emph{seminorm}, or \emph{prenorm}, on a group $G$ is a function $\|{\boldsymbol\cdot}\|\colon 
G\to \mathbb R$ such that $\|1\|=0$, $\|gh\|\le \|g\|+\|h\|$, and $\|g^{-1}\|=\|g\|$ for all $g, h\in G$. A 
seminorm which takes the value 0 only at the identity element is called a \emph{norm}. The topology of any 
topological group is determined by continuous seminorms in the sense that all open balls with respect to all 
continuous seminorms form a base of neighborhoods of the identity element (see \cite[Sec.~3.3]{AT}). 

In what follows, we consider the free, free Abelian, and free Boolean topological 
groups of a space $X$ in the sense of Graev \cite{Graev1948}; we denote them by $F_G(X)$, $A_G(X)$, and $B_G(X)$, 
respectively. Given a space $X$ in which an arbitrary point $x_0$ is fixed, the group $F_G(X)$ is the unique 
topological group with identity element $1=x_0$ containing $X$ as a subspace and characterized by the property 
that any continuous map $f$ of $X$ to any topological group $G$ that takes $x_0$ to the identity element of $G$ 
can be extended to a continuous homomorphism $F_G(X)\to X$. The group $F_G(X)$ does not depend on the point 
$x_0$: different choices of $x_0$ yield topologically isomorphic groups. 

Graev's definition differs from Markov's classical definition of the free topological group $F(X)$ of $X$ 
in that the identity element of 
$F(X)$ does not belong to $X$ and all continuous maps of $X$ to topological groups can be extended to continuous 
homomorphisms of $F(X)$. Graev's free groups are a generalization of Markov's ones in the sense that any free 
topological group in the sense of Markov is a free topological group in the sense of Graev 
($F(X)$ is isomorphic to $F_G(X\oplus\{e\})$, where $\{e\}$ is a singleton). 

The free Abelian (Boolean) topological group is defined in a similar way with the difference that it is 
Abelian (Boolean) and only continuous maps to Abelian (Boolean) topological groups are required to extend to 
continuous homomorphisms. Note that algebraically the free Boolean group generated by a set $X$ is nothing but 
the set $[X]^{<\omega}$ of all finite subsets of $X$ with the operation of symmetric difference. 
Detailed information on free, free Abelian, and free Boolean topological groups can be found in \cite{VINITI, 
Axioms, ArXiv}.

We use the standard notations $\omega$ for the set of nonnegative integers, $\mathbb R$ for the set of 
real numbers, $\overline A$ for the closure of a set $A$, $|A|$ for the cardinality of $A$,  
$\langle A\rangle $ for the subgroup generated by a subset $A$ of a group, $\operatorname{Fix} f$ for the fixed 
point set of a map $f$, and $\beta f$ for the continuous extension of a continuous map $f$ of a topological space 
to the Stone--\v Cech compactification of this space. By $\psi(X)$ we denote the pseudocharacter of a space $X$. 
A topological group is of countable pseudocharacter if and only if its identity element is a $G_\delta$-set. 

By $\dim X$ (by $\dim_0 X$) we denote the covering dimension of $X$ in the sense of \v Cech 
(in the sense of Kat\v etov), that is, 
the least integer $n\ge -1$ 
such that any finite open (cozero) cover of $X$ has a finite open (cozero) refinement of order $n$, provided that 
such an integer exists (if it does not exist, then the covering dimension is $\infty$). 
It is well known that $\dim_0 X=\dim_0 
\beta X$ (see, e.g., \cite[Theorem~11.10]{Charalambous}) and that $\dim X=\dim_0 X$ for normal spaces 
(see, e.g., \cite[Proposition~11.2]{Charalambous}). By a \emph{zero-dimensional} space we mean a space 
in which clopen sets form a base of topology, that is, a space $X$ with $\operatorname{ind} X=0$. 

The study of homogeneity in extremally disconnected and $F$-spaces heavily employs ultrafilters. In 
this paper we use \emph{selective}, or \emph{Ramsey}, ultrafilters on $\omega$. One of the equivalent definitions 
of a Ramsey ultrafilter $\mathscr U$ is as follows (see \cite[Proof of Lemma~9.2]{Jech}): for any family $\{A_n: 
n \in\omega\}$, where $A_n\in \mathscr U$, there exists its \emph{diagonal quasi-intersection} in $\mathscr U$, 
that is, a set $D\in \mathscr U$ such that $j \in A_i$ whenever $i, j \in D$ and $i < j$. Both the existence and 
the nonexistence of selective ultrafilters are consistent with ZFC \cite[p.~76]{Jech}. We also mention rapid 
ultrafilters (their other names are semi-$Q$-points and weak $Q$-points); for our considerations, it only 
matters that the nonexistence of rapid ultrafilters is consistent with ZFC~\cite{Miller}. 

\smallskip

We begin with the following simple observation.

\begin{remark}
\label{Remark0}
If there exist no rapid ultrafilters, then all countable subsets of any $F'$-group 
are discrete (and closed). 

Indeed, let $G$ be an $F'$-group, and let $X\subset G$ be countable. Then $H=\langle X\rangle $ is a countable 
subgroup of $G$. If $A$ and $B$ are any disjoint open subsets of $H$, then, according to \cite[3B.4]{GJ}, there 
are disjoint cozero sets $U\supset A$ and $V\supset B$ in $G$. We have $\overline A\cap \overline 
B = \varnothing$, because $G$ is an $F'$-space. Thus, $H$ is a countable extremally disconnected group, and the 
existence of a nondiscrete group with these properties implies that of rapid ultrafilters~\cite{RS}. 
\end{remark}


In what follows, we use the facts and observations listed below. All of them are either well known 
or obvious (or both). 

\begin{fact}
\label{Fact1}
Any countable union of cozero sets is a cozero set~\cite[1.14]{GJ}.
\end{fact}

\begin{fact}
\label{Fact2}
A space $X$ is an $F$-space if and only if so is $\beta X$~\cite[14.25]{GJ}.
\end{fact}

\begin{fact}
\label{Remark1}
Extremal disconnectedness, basic disconnectedness, and the property of being an $F'$-space are preserved by open 
continuous maps. {\rm (This easily follows from the equality  $\overline A=f(\overline{f^{-1}(A)})$, which holds 
for any  open map $f\colon X\to Y$ and any $A\subset Y$.)} 
\end{fact}

\begin{fact}
\label{Remark2}
If $G$ is a topological group, $H$ is its subgroup, and $G/H$ is the quotient space of left or right cosets, then 
the canonical quotient map $G\to G/H$ is open (see~\cite{AT}). 
\end{fact}

\begin{fact}
\label{Remark00}
The free Abelian topological group $A_G(X)$ is the topological 
quotient of $F_G(X)$ by the commutator subgroup, and the free Boolean topological group $B_G(X)$ 
is the topological quotient of $A_G(X)$ by the subgroup $A_G(2X)$ of squares. 
{\rm (For the case of $A_G(X)$, see \cite{Markov}. The case of $B_G(X)$ is similar.)}
\end{fact}

\begin{fact}
\label{fact_added}
If $Y$ is an $\mathbb R$-quotient space of  $X$, then the groups $F_G(Y)$, $A_G(Y)$, and $B_G(Y)$ are topological 
quotients of $F_G(X)$, $A_G(X)$, and $B_G(X)$, respectively (this was proved in \cite{Okunev} 
for the case of the free topological group; the remaining cases are similar). 
\end{fact}

\begin{fact}
\label{Remark3}
For any space $X$, the following conditions are equivalent: 
\begin{enumerate}
\item
$X$ is a $P$-space; 
\item $F_G(X)$ is a $P$-space; 
\item
$A_G(X)$ is a $P$-space; 
\item 
$B_G(X)$ is a $P$-space.
\end{enumerate}
{\rm (To show (1), it suffices to note that if $X$ is a $P$-space, then all $G_\delta$-sets in $F_G(X)$ form a 
group topology on the free group which is finer than the topology of $F_G(X)$ but still induces the original 
topology of $X$ on $X$. Since the topology of $F_G(X)$ is the finest group topology with the latter property, 
it follows that all $G_\delta$-sets are open in $F_G(X)$. Obviously, the property of being a $P$-space is 
hereditary. The assertions about $A_G(X)$ and $B_G(X)$ are proved in a similar way.)} 
\end{fact}

\begin{remark}
\label{Remark4}
If a topological group $G$ is not a $P$-space, then there exist neighborhoods $ U_i$, $i\in \omega$, of 
the identity element  such that $U_{n+1}\cdot U_{n+1}\subset U_n$ and $U_n=U_n^{-1}$ for all 
$n\in \omega$ and the identity element is not in the interior of the intersection $H=\bigcap_{n\in \omega}U_n$. 
The set $H$ is closed (because if $x\notin U_n$ for some $n\in \omega$, then 
$x\cdot U_{n+1}\cap U_{n+1}=\varnothing$), and this is a subgroup (by construction). Clearly, any subgroup 
with nonempty interior must be open; therefore, $H$ is a nowhere dense closed subgroup of~$G$. 

If, in addition, 
$G$ is Lindel\"of, then, for each $n$, there exists neighborhoods $V_{n,i}$, $i\in \omega$, of 
the identity element  with the following properties: (a)~$V_{n,0}\subset U_n$, $V_{n,i+1}\cdot 
V_{n,i+1}\subset V_{n,i}$, and $V_{n,i}=V_{n,i}^{-1}$ for all $n\in \omega$; (b)~for any $x\in G$ and any $i\in 
\omega$, there exists a $j\in \omega$ such that $x^{-1}\cdot V_{n,j}\cdot x\subset V_{n,i}$ 
(see \cite[Propositions~3.4.6 and~3.4.10, Lemma~3.4.14]{AT}). Setting $V_n=\bigcap_{k,i\le n}V_{k,i}$ for $n\in 
\omega$, we obtain a sequence of neighborhoods $V_n$ of the identity element such that 
$V_{n+1}\cdot V_{n+1}\subset V_n$ and $U_n=U_n^{-1}$ for all 
$n\in \omega$ and $N=\bigcap_{n\in \omega}V_n$ is a nowhere dense closed normal subgroup of~$G$.  
\end{remark}

The following theorem is the first main result of this paper.

\begin{theorem}
\label{Theorem1}
Any paracompact topological $F$-group $G$ such that $\dim G<\infty$ and $\psi (G)\le \omega$ contains 
an open Boolean subgroup with the same properties. 
\end{theorem}

\begin{proof}
Consider the automorphism $h\colon G\to G$ defined by $h(x)=x^{-1}$ for $x\in G$. Extending it to $\beta G$, we 
obtain an autohomeomorphism $\beta h\colon \beta G\to \beta G$ which takes $\beta G\setminus G$ to $\beta 
G\setminus G$. Since $\dim_0 \beta G <\infty$ and $\beta G$ is a compact $F$-space, it follows that 
$\operatorname{Fix} \beta h$ is a $P$-set in $\beta G$ \cite{Hart-Vermeer}. 
In particular, any zero set in $\beta G$ containing $\operatorname{Fix} \beta h$ is a neighborhood of 
$\operatorname{Fix} \beta h$. 

Using the assumption $\psi(G)\le \omega$, we can find a sequence $(U_n)_{n\in \omega}$ of 
neighborhoods of $0$ in $G$ such that $ \bigcap_{n\in \omega} U_n=\{0\}$, $U_{n+1}\cdot U_{n+1}\subset U_n$, and 
$U_n=U^{-1}_n$ for all $n\in \omega$. There exists a norm $\|{\boldsymbol\cdot}\|$ on $G$ such that 
$$
\{x\in G: \|x\|<1/2^n\}\subset U_n\subset \{x\in G: \|x\|\le 2/2^{n}\}
$$ 
for every $n\in \omega$ (see, e.g., \cite[Lemma~3.3.10]{AT}). Consider the continuous function 
$\varphi\colon G\to \mathbb R$ defined by $\varphi(x)=\|x^2\|$ for $x\in G$. Note that 
$\operatorname{Fix} h=\varphi^{-1}(\{0\})$. 

Let $F=\varphi^{-1}(\{0\})$, and let $C=G\setminus F$. Since $C$ is a cozero set (and hence an $F_\sigma$-set)  
in the paracompact $F$-space $G$, it follows that $C$ is paracompact \cite[Theorem~5.1.28]{Engelking} and 
$C^*$-embedded in $G$ \cite[Theorem~14.25]{GJ} (the latter implies that 
$\beta C$ is the closure $\overline C$ of $C$ in $\beta G$). 
According to Corollary~11.21 in \cite{Charalambous}, we have $\dim C <\infty$. Therefore, the extension 
$\beta (h|_C)$ of the  fixed-point free 
autohomeomorphism $h|_C$ to $\overline C=\beta C$ has no fixed points~\cite{vanDouwen}.  
It follows that $F$ is open in~$G$. 

Thus, $F$ is an open neighborhood of 1 in $G$. Let $U$ be an open neighborhood of 1 such that $U^2\subset F$. 
Then the subgroup $\langle U\rangle $ generated by $U$ is Boolean. Indeed, if $x,y\in \langle U\rangle$, then 
$x,y,x\cdot y\in F$, whence $x\cdot y= y^{-1}\cdot x^{-1}=y\cdot x$. Thus, the subgroup $\langle U\rangle$ is 
Abelian. Since it is generated by elements of order 2 and has nonempty interior, it easily follows that 
$\langle U\rangle$ is an open (and hence clopen) Boolean subgroup of $G$.
It remains to note that all properties of 
$G$ listed in the statement of the theorem are inherited by clopen subspaces.
\end{proof}

\begin{corollary}
The existence of a nondiscrete paracompact topological $F$-group $G$ with $\dim G<\infty$ and $\psi(G)\le 
\omega$ is equivalent to the existence of a nondiscrete Boolean topological group 
with the same properties. 
\end{corollary}

\begin{theorem}
\label{Theorem2}
Any Lindel\"of basically disconnected topological group either is a $P$-space 
or has a nondiscrete topological 
quotient of countable pseudocharacter containing an open basically disconnected Boolean subgroup. 
\end{theorem}

\begin{proof}
Let $G$ be a Lindel\"of basically disconnected group. If $G$ is not a $P$-space, 
then by Remark~\ref{Remark4} \,$G$ contains 
a closed nowhere dense $G_\delta$ normal subgroup $N=\bigcap_{n\in \omega}V_n$, where 
$V_i$, $i\in \omega$, are neighborhoods of the identity element 
such that  $V_{n+1}\cdot V_{n+1}\subset V_n$ and $V_n=V^{-1}_n$ for all $n\in \omega$. We have 
$N=\bigcap_{n\in \omega}(V_n\cdot N)$. Indeed, if $x\in G\setminus N$, then $x\notin V_n$ for some $n$, so that 
$x\notin V_{n+1}\cdot V_{n+1}\supset V_{n+1}\cdot N$. By Fact~\ref{Remark2} the 
canonical quotient map $h\colon G\to G/N$ is open, and by Fact~\ref{Remark1} the quotient $G/N$ is basically 
disconnected. It is nondiscrete, because $N$ is nowhere dense in $G$, and 
$$ 
\bigcap_{n\in \omega} 
h(V_n)=h\Bigl(\bigcap_{n\in \omega} h^{-1}(V_n)\Bigr)= h\Bigl(\bigcap_{n\in \omega} 
(V_n\cdot H)\Bigr)=h(H)=\{1\}\quad\text{in $G/N$}. 
$$ 
Therefore, $\psi(G/N)\le \omega$.  Finally, $G/N$ is Lindel\"of and $\dim G/N=0$, 
because the Stone--\v Cech compactification of any basically disconnected space is obviously basically 
disconnected and zero-dimensional. Thus, $G/N$ satisfies all assumptions of Theorem~\ref{Theorem1}. 
\end{proof}

\begin{corollary}
The existence of a Lindel\"of  basically disconnected group which is not a $P$-space 
is equivalent to the existence 
of a nondiscrete Lindel\"of Boolean basically disconnected group of countable pseudocharacter. 
\end{corollary}

Our next theorem is concerned with free, free Abelian, and free Boolean topological $F'$-groups. Its proof is 
based on the following statements. 

\begin{proposition}
\label{pr1}
Suppose that a space $X$ contains clopen subsets $U_n$, $n\in \omega$,  such that $U_{n+1}\subset U_n$ for 
$n\in \omega$ and $C=\bigcap_{n\in \omega}U_n$ is a nonopen nonempty set. Then there exists a nondiscrete 
countable space $Y$ such that the groups $F_G(Y)$, $A_G(Y)$, and $B_G(Y)$ are topological quotients of $F_G(X)$, 
$A_G(X)$, and $B_G(X)$, respectively. 
\end{proposition}

\begin{proof}
We set $C_0=X\setminus U_0$ and $C_n=U_{n-1}\setminus U_n$ for 
$n=1,2,\dots$ and  let $Y$ be the image of $X$ under the quotient map contracting $C$ and each $C_n$, $n=0, 
1,\dots$, to a point. Clearly, $Y$ is a countable completely regular Hausdorff space with only one nonisolated 
point (the image of $C$). It remains to recall Fact~\ref{fact_added}. 
\end{proof}

\begin{proposition}
\label{pr2}
Let $X$ be a space, and let $x_0$ and $y_0$ be non-$P$-points in $X$. Then there exists a 
Tychonoff space $Y$ and an $\mathbb R$-quotient map $f\colon X\to Y$ such that $Y$ is a subset of $\mathbb R$ 
endowed with a topology finer than that induced by the Euclidean metric of~$\mathbb R$ and the points 
$f(x_0)$ and $f(y_0)$ are non-$P$-points in $Y$. Moreover, if $x_0\ne y_0$, then $f$ can be chosen so that 
$f(x_0)<f(y_0)$. 
\end{proposition}

\begin{proof}
Using the complete regularity of $X$, it is easy to construct a continuous function 
$f\colon X\to \mathbb R$ such that $f(x_0)<f(y_0)$, $x_0$ does not belong to the interior of 
$f^{-1}(\{f(x_0)\})$, $y_0$ does not belong to the interior of $f^{-1}(\{f(y_0)\})$, and  if 
$x_0\ne y_0$, then $f(x_0)<f(y_0)$. Let $Y$ be the set $f(X)$ endowed with the $\mathbb R$-quotient topology 
with respect to $f$. Since this is the finest completely regular topology with respect to which $f$ is 
continuous, it follows that the topology of $Y$ is Tychonoff and finer than that induced from~$\mathbb R$. 
That the points $f(x_0)$ and $f(y_0)$ are not isolated in $Y$, because their preimages are not open in $X$. 
Therefore, both of them are non-$P$-points in $Y$.
\end{proof}

\begin{theorem}
\label{pr3}
\begin{enumerate}
\item 
Any space $X$ for which $B_G(X)$ is an $F'$-group contains at most one non-$P$-point.
\item
If a space $X$ is not a $P$-space and $B_G(X)$ is an $F'$-group, then there exists a countable space $Z$ with 
a unique nonisolated point such that $Z$ is an $\mathbb R$-quotient image of $X$ and $B_G(Z)$ is an extremally 
disconnected quotient of $B_G(X)$. 
\end{enumerate}
\end{theorem}

\begin{proof}
(1)\enspace 
Suppose that $x_0$ and $y_0$ are two different non-$P$-points of $X$. 
Let $Y$ and $f\colon X\to Y$ be as in Proposition~\ref{pr2}. 
According to Facts~\ref{Remark2} and~\ref{fact_added} the 
group $B_G(Y)$ is an open image of $B_G(X)$, and by Fact~\ref{Remark1} it is an $F'$-group. 

We denote by $d$ the continuous 
metric on $Y$ induced by the Euclidean metric and by $\|{\boldsymbol\cdot}\|_d$ the Graev extension of 
$d$ to a continuous seminorm on $B_G(Y)$ (see~\cite{Axioms}). 
Let $d(f(x_0),f(y_0))=a$. We have $a>0$. The sets 
$$
U=\{x\in B_G(Y): \|f(x_0)-x\|_d<a/2\}
$$
and
$$
V=\{x\in B_g(Y): \|f(x_0)-x\|_d>a/2\}
$$
are disjoint cozero sets in $B_G(Y)$. Therefore, they have disjoint closures. 

Let $z_0\in \mathbb R$ be the midpoint between $f(x_0)$ and $f(y_0)$, that is, $z_0=f(x_0)+a/2=f(y_0)- 
a/2$. We set 
$$
\widetilde U=\{y\in Y: f(x_0)-a/2 <y<z_0\} \quad\text{and}\quad 
\widetilde V=\{y\in Y: z_0<y<f(y_0)+a/2\}.
$$ 
Note that 
$U\cap Y= \widetilde U$ and $V\cap Y= \widetilde V$. Therefore, either $z_0\notin \overline U\cap Y$ (in which 
case the sets $\{y\in Y: y<z_0\}$ and $\{y\in Y: z_0\le y\}$ are disjoint closed subsets of $Y$ covering $Y$) 
or $z_0\notin \overline V\cap Y$ (in which 
case such sets are $\{y\in Y: y\le z_0\}$ and $\{y\in Y: z_0< y\}$). In either case, $Y$ has a clopen subset $W$ 
containing $f(x_0)$ and missing $f(y_0)$. Thus, $Y$ is the topological sum $W\oplus Y\setminus W$, whence 
$B_G(Y)=B_G(W)\times B_G(Y\setminus W)$  \cite[Proposition~7]{Axioms}. We have shown that the $F'$-space $B_G(Y)$ 
is the product of two spaces each of which contains a non-$P$-point and hence is not a $P$-space. This 
contradicts the main theorem of~\cite{Curtis}. 

(2)\enspace
Let $X$ be a non-$P$-space for which the free Boolean topological group $B_G(X)$ 
is an $F'$-space. By Proposition~\ref{pr2} \,$X$ has 
a nondiscrete $\mathbb R$-quotient $Y$ of countable pseudocharacter. 
By Facts~\ref{Remark1}, \ref{Remark2}, and~\ref{fact_added} \,$B_G(Y)$ is an $F'$-group. Since the 
pseudocharacter of $Y$ is countable, it follows from assertion~(1) that $Y$ has only one 
nonisolated point, and this point is not a $P$-point. Proposition~\ref{pr1} implies the existence of a 
nondiscrete countable space $Z$ such that $B_G(Z)$ is a topological quotient of $B_G(Y)$ and hence an $F'$-group; 
therefore, $B_G(Z)$ is extremally disconnected. 
\end{proof}

\begin{corollary}
The existence of a free Boolean topological $F'$-group which is not a $P$-space is 
equivalent to the existence of a selective ultrafilter on~$\omega$. 
\end{corollary}

\begin{proof}
If there exists a  non-$P$-space $X$ for which $B_G(X)$ is an $F'$-group, then, by Theorem~\ref{pr3}\,(2), there 
exists a nondiscrete countable space $Z$ for which $B_G(Z)$ is extremally disconnected. According to 
\cite{Sipa-free-ed}, the existence of a nondiscrete free Boolean extremally disconnected group implies that of a 
selective ultrafilter on~$\omega$. 

Conversely, it is well known (see, e.g., \cite[Theorem~5.1]{Zelenyuk} or 
\cite[Theorem~8.2]{ArXiv}) that the existence of a selective ultrafilter on $\omega$ implies the existence of a 
nondiscrete countable free Boolean topological group which is an extremally disconnected space and hence an 
$F'$-space. Clearly, being countable and nondiscrete, it cannot be a $P$-space.
\end{proof}


\begin{theorem}
\label{Theorem3}
For any space $X$,  the following conditions are equivalent: \textup{(i)}~the free 
topological group of $X$ is an $F'$-space, \textup{(ii)}~the free Abelian 
topological group of $X$ is an $F'$-space, \textup{(iii)}~$X$ is a $P$-space. 
\end{theorem}

\begin{proof}
Let $X$ be a non-$P$-space. In view of Facts~\ref{Remark1}--\ref{Remark00} it suffices to check that $A_G(X)$ 
is not an $F'$-group. Assume the contrary. Then $B_G(X)$ is an $F'$-group by Fact~\ref{Remark00}. By  
Theorem~\ref{pr3}\,(2) there exists a nondiscrete countable $\mathbb R$-quotient $Z$ of $X$. By 
Fact~\ref{fact_added} \,$A_G(Z)$ is an $F'$-group. It is extremally disconnected, being countable. 
According to Malykhin's theorem, any extremally disconnected group contains an open Boolean 
subgroup~\cite{Malykhin1975}. However, the only Boolean subgroup of any free Abelian group is trivial. 
Thus, $A_G(Z)$ must be discrete, which contradicts the nondiscreteness of~$Z$. 
\end{proof}

The author is most grateful to Evgenii Reznichenko for very fruitful discussions.


\begin{thebibliography}{0}



\bibitem{AT}
A. Arhangel'skii and M.~Tkachenko,
\textit{Topological Groups
and Related Structures} (Atlantis Press/World Sci., Amsterdam--Paris, 2008).

\bibitem{Charalambous} 
M. G. Charalambous,
\textit{Dimension Theory:
A Selection of Theorems
and Counterexamples} 
(Springer International, Cham, 2019).

\bibitem{Curtis}
P.~C.~Curtis, Jr.,
``A note concerning certain product spaces,''
Arch. Math. \textbf{11}, 50--52 (1960).

\bibitem{vanDouwen}
E. K. van Douwen,
``$\beta X$ and fixed-point free maps,''
Topol. Appl. \textbf{51}, 191--195  (1993).

\bibitem{Engelking}
R. Engelking,
General Topology, 
Heldermann, Berlin, 1989. 

\bibitem{GJ}
L. Gillman and M. Jerison, 
Rings of Continuous Functions, 
Springer, New York, 1960.

\bibitem{Graev1948}
M.~I. Graev,
``Free topological groups,''
Izv. Akad. Nauk SSSR. Ser. Mat.,
\textbf{12}, 279--324 (1948).

\bibitem{Graev1950}
M.~I. Graev,
``The theory of topological groups I,''
Usp. Mat. Nauk,
\textbf{5} (2),
3--56 (1950).

\bibitem{Hart-Vermeer}
K.~P.~Hart and J.~Vermeer,
``Fixed-point sets of autohomeomorphisms of compact $F$-spaces,''
Proc. Amer. Math. Soc. \textbf{123} (1),   311--314 (1995). 

\bibitem{Hindman}
W. W. Comfort, Neil Hindman, and S.~Negrepontis,
``$F$-Spaces and their product with $P$-spaces,''
Pacif. J. Math. \textbf{28} (3),  (1969).

\bibitem{Jech}
T. Jech, 
\textit{Set Theory}, 3rd ed.\  (Springer, 
Berlin, 2003). 

\bibitem{R-quotient}
S.~M.~Karnik and S.~Willard,
``Natural covers and $R$-quotient mappings,''
Canad. Math. Bull. \textbf{25} (4), 456--462 (1982).

\bibitem{Malykhin1975}
V. I. Malyhin, 
``Extremally disconnected and similar groups,''
Dokl. Akad. Nauk SSSR \textbf{220}, 27--30 (1975); English transl.:
Soviet Math. Dokl. \textbf{16}, 
21--25 (1975).

\bibitem{Malykhin1979}
V. I. Malykhin,
``On extremally disconnected topological groups,''
Usp. Mat. Nauk \textbf{34} (6(210)), 59--66 (1979);
English transl.
Russian Math. Surveys \textbf{34} (6), 67--76 (1979).

\bibitem{Markov}
A.~A.~Markov,
``On free topological groups,''
Izv. Akad. Nauk SSSR, Ser. Mat.
\textbf{9} (1),
3--64 (1945); English transl.:
Amer. Math. Soc. Transl. \textbf{30},  11--88 (1950); 
Reprint: Amer. Math. Soc. Transl. \textbf{8} (1), 195--272 (1962).

\bibitem{Miller}
A. W. Miller,  
``There are no $Q$-points in Laver's model for the Borel 
conjecture,'' Proc. Amer. Math. Soc., \textbf{78} (1), 103--106, 1980. 

\bibitem{Okunev}
O. G. Okunev,
``A method for constructing examples of $M$-equivalent spaces,''
Topol.  Appl. \textbf{36}, 157--171 (1990).

\bibitem{Protasov}
I. Protasov, 
``Filters and topologies on semigroups,'' 
Mat. Studii \textbf{3},
15--28 (1994).

\bibitem{R20}
E. Reznichenko, 
Homogeneous subspaces of products of extremally disconnected spaces, 
Topol. Appl. 284 (2020) 107403.

\bibitem{RS}
E. Reznichenko and O. Sipacheva, 
``Discrete subsets in topological groups
and countable extremally disconnected groups,''
Proc. Amer. Math. Soc. \textbf{149}, 2655--2668 (2021).

\bibitem{VINITI}
O.V. Sipacheva, 
``The topology of free topological groups,'' 
Fundam. Prikl. Math. \textbf{9} (2), 99--204 (2003); English transl.:
J. Math. Sci. 131 (4) (2005) 5765--5838.

\bibitem{Sipa-free-ed}
O. V. Sipacheva,
``The nonexistence of extremally disconnected free topological groups,''
Topol. Appl. \textbf{160}, 1227--1231  (2013).

\bibitem{Axioms}
O. Sipacheva, 
``Free Boolean topological groups,'' 
Axioms \textbf{4} (4),  492--517 (2015).

\bibitem{ArXiv}
Ol'ga Sipacheva, ``Free Boolean Topological Groups,'' arXiv:1612.04878 [math.GN].

\bibitem{A1}
Ol'ga Sipacheva, 
``Extremally disconnected groups of measurable cardinality,'' arXiv:2104.11822 [math.GN].

\bibitem{Zelenyuk}
Y. G. Zelenyuk, 
\textit{Ultrafilters and Topologies on Groups} (De Gruyter, Berlin--New York, 2011).



\end{thebibliography}
\end{document}